\documentclass[bibalpha]{amsart}
\usepackage{amssymb}
\usepackage{textcomp}
\usepackage[mathscr]{euscript}
\usepackage{pb-diagram}
\usepackage{enumitem}
\usepackage{comment}

\newtheorem{theorem}{Theorem}[section]

\newtheorem{thm}[theorem]{Theorem}
\newtheorem{prop}[theorem]{Proposition}
\newtheorem{lem}[theorem]{Lemma}

\newtheorem{fact}[theorem]{Fact}
\newtheorem{cor}[theorem]{Corollary}

\theoremstyle{definition}

\theoremstyle{remark}

\newcommand{\Log}{\operatorname{Log}}
\newcommand{\anfield}{\Q^{\mathrm{an}}_p}
\newcommand{\valp}{\operatorname{val_p}}
\newcommand{\tp}{\operatorname{tp}}
\newcommand{\val}{\operatorname{val}}

\newcommand{\rfield}{(\R,+,\times}

\newcommand{\pfield}{\Q_p}

\newcommand{\mfrak}{\mathfrak{m}}

\ProvideTextCommandDefault{\cprime}{(U+042C)}

\newcommand{\sub}{\subseteq}

\newcommand{\st}{\operatorname{st}}

\newcommand{\slam}{(\mathscr{R},\lambda^{\mathbb{Z}})}

\newcommand{\Sh}[1]{\ensuremath{\mathscr{#1}^{\mathrm{Sh}}}}
\newcommand{\Sq}[1]{\ensuremath{\mathscr{#1}^{\square}}}
\newcommand{\nip}{\mathrm{NIP}}

\newcommand{\Cal}[1]{\ensuremath{\mathcal{#1}}}
\newcommand{\Sa}[1]{\ensuremath{\mathscr{#1}}}

\newcommand{\Z}{\mathbb{Z}}
\newcommand{\N}{\mathbb{N}}

\newcommand{\Q}{\mathbb{Q}}
\newcommand{\R}{\mathbb{R}}

\begin{document}

\title[]{Nippy proofs of p-adic results of Delon and Yao}

\author{Erik Walsberg}
\address{Department of Mathematics, Statistics, and Computer Science\\
Department of Mathematics\\University of California, Irvine, 340 Rowland Hall (Bldg.\# 400),
Irvine, CA 92697-3875}
\email{ewalsber@uci.edu}
\urladdr{http://www.math.illinois.edu/\textasciitilde erikw}

\date{\today}

\maketitle

\begin{abstract}
Let $K$ be an elementary extension of $\Q_p$, $V$ be the set of finite $a \in K$, $\st$ be the standard part map $K^m \to \mathbb{Q}^m_p$, and $X \subseteq K^m$ be $K$-definable.
Delon has shown that $\mathbb{Q}^m_p \cap X$ is $\mathbb{Q}_p$-definable.
Yao has shown that $\dim \mathbb{Q}^m_p \cap X \leq \dim X$ and $\dim \st(V^n \cap X) \leq \dim X$.
We give new $\nip$-theoretic proofs of these results and show that both inequalities hold in much more general settings.
We also prove the analogous results for the expansion $\anfield$ of $\mathbb{Q}_p$ by all analytic functions $\mathbb{Z}^m_p \to \mathbb{Q}_p$.
As an application we show that if $(X_k)_{k \in \N}$ is a sequence of elements of an $\anfield$-definable family of subsets of $\Q^m_p$ which converges in the Hausdroff topology to $X \subseteq \Q^m_p$ then $X$ is $\anfield$-definable and $\dim X \leq \limsup_{k \to \infty} \dim X_k$.
\end{abstract}

\section{Introduction}
\noindent
Fix a prime $p$, let $K$ be an elementary extension of $\pfield$, $V$ be the set of $a \in K$ such that $\val(a) \geq k$ for some $k \in \Z$, and $\st$ be the standard part map $V^m \to \Q^m_p$.
Fact~\ref{fact:delon}, a $p$-adic analogue of the Marker-Steinhorn theorem~\cite{Marker-Steinhorn}, is due to Delon~\cite{Delon-def}.

\begin{fact}
\label{fact:delon}
If $X \subseteq K^m$ is $K$-definable then $\Q^m_p \cap X$ is $\pfield$-definable.
\end{fact}

\noindent
Fact~\ref{fact:residue-0} follows from standard results on equicharacteristic zero Henselian valued fields as $V$ is a Henselian valuation ring.

\begin{fact}
\label{fact:residue-0}
If $X \subseteq K^m$ is $K$-definable then $\st(V^m \cap X)$ is $\pfield$-definable.
\end{fact}

\noindent
Let $\dim X$ be the dimension of a $\pfield$-definable set $X$.
Fact~\ref{fact:yao} is a result of Yao~\cite{yao}.

\begin{fact}
\label{fact:yao}
Suppose that $X \subseteq K^m$ is $\Sa K$-definable.
Then
$ \dim \Q^m_p \cap X \leq \dim X $
and
$ \dim \st(V^m \cap X) \leq \dim X.$
\end{fact}

\noindent
As $\Q^m_p \cap X$ is a subset of $\st(V^m \cap X)$ the first inequality is a corollary to the second, but we will see that they generalize in different directions.
The analogue of Fact~\ref{fact:yao} for o-minimal expansions of $\rfield)$ were previously proven by van den Dries~\cite{t-convexityII}.
We show that Fact~\ref{fact:yao} and its o-minimal analogue follow easily from the theory of externally definable sets in $\nip$ structures.
The first inequality generalizes to an arbitrary elementary extension of an arbitrary $\nip$ structure and the second inequality generalizes to any ``tame extension" of dp-minimal valued fields.
We also show that Fact~\ref{fact:delon} follows from general $\nip$ results and Fact~\ref{fact:residue-0}.
In Section~\ref{section:analytic} we prove the analogues of Facts~\ref{fact:delon} and \ref{fact:yao} for $\anfield$.
In Section~\ref{section:geometric} we use these results to study Hausdorff limits of $\anfield$-definable sets, this is the $p$-adic analogue of o-minimal work of van den Dries~\cite{vdd-limit}.

\subsection{Acknowledgements}
Thanks to Raf Cluckers and Silvian Rideau for providing a reference for Fact~\ref{fact:sr}.

\section{Conventions and notation}
\noindent
By ``definable" we mean ``first order definable, possibly with parameters".
Throughout $m,n$ are natural numbers, $i,j,k,l$ are integers, and $\Sa M$ is an $L$-structure.
If $X$ is an $\Sa M$-definable set and $\Sa M \prec \Sa N$ then $X(\Sa N)$ is the $\Sa N$-definable set defined by the same formula as $X$.
Two structures on the same domain are \textbf{interdefinable} if they define the same sets.
If $x = (x_1,\ldots,x_n)$ is a tuple of variables then $|x| = n$.
The structure \textbf{induced} on $A \subseteq M^m$ by $\Sa M$ is the structure with domain $A$ and an $n$-ary relation for $A^n \cap X$ for each $\Sa M$-definable $X \subseteq M^{nm}$.
Our reference on $\nip$ and dp-rank is Simon's book~\cite{Simon-Book}.

\section{Dp rank}
\label{section:dp-rank}
\subsection{Definition}
To generalize Fact~\ref{fact:yao} we need a general notion of dimension.
We use dp-rank.
The dp-rank of a definable set is either a cardinal or the formal symbol $\infty$ which is by declared to be larger than all cardinals.
Fix an $\Sa M$-definable set $X \subseteq M^{|y|}$.
Let $\lambda$ be a cardinal.
An $(\Sa M,X,\lambda)$-array consists of a sequence $( \varphi_\alpha(x_\alpha;y) : \alpha < \lambda)$ of $L$-formulas and an array $( a_{\alpha,i} \in M^{|x_\alpha|} : \alpha < \lambda, i < \omega)$ such that for every $f : \lambda \to \omega$ there is $b \in X$ such that 
$$ \Sa M \models \varphi_\alpha(a_{\alpha,i};b) \quad \text{if and only if} \quad f(\alpha) = i \quad \text{for all   } \alpha,i. $$
Then $\dim X \geq \lambda$ if there is $\Sa M \prec \Sa N$ and an $(\Sa N,X(\Sa N),\lambda)$-array.
If $\dim X \geq \lambda$ for all cardinals $\lambda$ then $\dim X := \infty$, we let
$ \dim X := \max\{ \lambda : \dim X \geq \lambda\} $
when this maximum exists and otherwise declare
$$ \dim X := \sup \{ \lambda : \dim X \geq \lambda \} - 1. $$
The dp-rank of $\Sa M$ is defined to be $\dim M$.
Of course these definitions raise the question of what exactly $\kappa - 1$ is when $\kappa$ is an infinite cardinal.
There are several options, and it does not matter which we select.
When the structure may not be clear from context we let $\dim_{\Sa M} X$ be the dp-rank of an $\Sa M$-definable set $X$.
\newline

\noindent
We will want to avoid passing to an elementary extension, so we use finitary arrays.
Let $\Phi$ be a sequence $(\varphi_\alpha(x_\alpha;y) : \alpha < \lambda)$ of $L$-formulas, $F \subseteq \lambda$ be finite, and $n \in \N$.
An $(\Sa M,X,\Phi,F,n)$-array is an array $(a_{\alpha,i} \in M^{|x_\alpha|} : \alpha \in F, i \leq n)$ such that for every $f : F \to n$ there is $b \in X$ such that for all $\alpha \in F, i \leq n$ we have $\Sa M \models \varphi_\alpha(a_{\alpha,i};b)$ if and only if $f(\alpha) = i$.
So $\dim X \geq \lambda$ if and only if there is such a $\Phi$ so that for every finite $F \subseteq \lambda$ and $n$ there is an $(\Sa M,X,\Phi,F,n)$-array.

\subsection{Properties}
\noindent
Dp-rank characterizes $\nip$ structures, see \cite{Simon-Book}.
\begin{fact}
\label{fact:dp-nip}
Let $T$ be a complete theory and $\Sa M \models T$.
The following are equivalent.
\begin{enumerate}
\item $\Sa M$ is $\nip$,
\item $\dim \Sa M < \infty$,
\item $\dim \Sa M < |T|^+$,
\item $\dim X < |T|^+$ for all $\Sa M$-definable sets $X$.
\end{enumerate}
\end{fact}

\noindent
Fact~\ref{fact:dp-rank-basic-1} shows that dp-rank is a reasonable notion of dimension.
The first two items follow easily from the definition and the fourth is proven in \cite{dp-rank-additive}.

\begin{fact}
\label{fact:dp-rank-basic-1}
Suppose $\Sa M$ is $\nip$,  $X,Y$ are definable sets, and $f : X \to M^m$ is definable.
\begin{enumerate}
\item $\dim X = 0$ if and only if $X$ is finite.
\item $\dim X \cup Y = \max \{\dim X , \dim Y\}$. 
(so $X \subseteq Y$ implies $\dim X \leq \dim Y$),
\item $\dim f(X) \leq \dim X$,
\item If $\dim f^{-1}(a) \leq \lambda$ for all $a \in f(X)$ then $\dim X \leq \dim f(X) + \lambda$.
\end{enumerate}
\end{fact}

\noindent
We say that $\Sa M$ is \textbf{dp-minimal} when $\dim \Sa M \leq 1$.
O-minimal structures and $\Q_p$ are both dp-minimal~\cite{dp-basic}.
Dp-rank is the canonical notion of dimension for definable sets in dp-minimal expansions of valued fields or divisible ordered abelian groups.
Fact~\ref{fact:sw} is proven in \cite{SW-tametop}.
Let $Y$ be a topological space and equip $Y^n$ with the product topology.
The \textbf{naive dimension} of a nonempty $X \subseteq Y^n$ is the  maximal $0 \leq k \leq n$ such that $\pi(X)$ has interior for some coordinate projection $\pi : Y^n \to Y^k$.
Acl-dimension is defined in the same way as dimension is usually defined in a geometric structure (this definition makes sense in any structure).

\begin{fact}
\label{fact:sw}
Let $\Sa M$ be a dp-minimal expansion of a valued field or a divisible ordered abelian group and $X \subseteq M^n$ be definable and nonempty.
The following are equal:
\begin{enumerate}
\item The dp-rank of $X$,
\item The $\mathrm{acl}$-dimension of $X$, and
\item The naive dimension of $X$.
\end{enumerate}
\end{fact}

\noindent
So in particular dp-rank agrees with the canonical dimension for definable sets in a $p$-adically closed field or an o-minimal expansion of an ordered abelian group.
It follows from Fact~\ref{fact:sw} that in this setting the dp-rank of $X$ depends only on $X$ and the topology, not on $\Sa M$.
We will also apply
Fact~\ref{fact:constructible}, proven in \cite{SW-tametop}.

\begin{fact}
\label{fact:constructible}
Suppose that $\Sa K$ is a dp-minimal expansion of a valued field.
Then every $\Sa K$-definable set is a boolean combination of closed $\Sa K$-definable sets.
\end{fact}

\section{Externally definable sets}
\noindent
Throughout this section $\Sa M \prec \Sa N$, $\Sa N$ is highly saturated, and $X \sub M^n$.
We say that $X$ is \textbf{externally definable} if $X = M^n \cap Y$ for some $\Sa N$-definable $Y$.
By saturation the collection of externally definable sets does not depend on choice of $\Sa N$.
We say that $\Sa M$ is \textbf{Shelah complete} if every externally definable set is definable.
The \textbf{Shelah completion} $\Sh M$ of $\Sa M$ is the structure induced on $M$ by $\Sa N$.
We say that $Y \subseteq N^n$ is an \textbf{honest definition} of $X$ if $Y$ is $\Sa N$-definable, $M^n \cap Y = X$, and whenever $Z \subseteq M^n$ is $\Sa M$-definable such that $Z \cap X = \emptyset$ then $Z(\Sa N) \cap Y = \emptyset$.
The second claim of Fact~\ref{fact:shelah-cs} is a theorem of Shelah~\cite{Shelah-external}.
The first is due to Chernikov and Simon~\cite{CS-I}.
The second claim is a corollary to the first.

\begin{fact}
\label{fact:shelah-cs}
Suppose $\Sa M$ is $\nip$.
Every externally definable subset has an honest definition.
Every $\Sh M$-definable set is externally definable.
\end{fact}

\noindent
It follows easily from Fact~\ref{fact:shelah-cs} that the Shelah completion of an $\nip$ structure is Shelah complete, this justifies our terminology.
Shelah observed that Fact~\ref{fact:shelah-cs} implies the first claim of Fact~\ref{fact:shelah-1}.
The second claim is due to Onshuus and Usvyatsov\cite{OnUs}.

\begin{fact}
\label{fact:shelah-1}
If $\Sa M$ is $\nip$ then $\Sh M$ is $\nip$.
If $\Sa M$ is dp-minimal then $\Sh M$ is dp-minimal.
\end{fact}

\noindent
The first claim of Fact~\ref{fact:external-basic} is elementary.
The second claim follows from the first, Fact~\ref{fact:shelah-cs}, and saturation.

\begin{fact}
\label{fact:external-basic}
Suppose $\Sa M \prec \Sa O$.
If $X \subseteq O^n$ is externally definable in $\Sa O$ then $M^n \cap X$ is externally definable in $\Sa M$.
If $\Sa M$ is $\nip$ then the structure induced on $M$ by $\Sh O$ is interdefinable with $\Sh M$.
\end{fact}

\noindent
Lemma~\ref{lem:completion} is easy and left to the reader.

\begin{lem}
\label{lem:completion}
Suppose that $\Sa M$ is $\nip$, $\Sa M$ is Shelah complete, and $\Sa O \prec \Sa M$.
Then every $\Sh O$-definable set is of the form $O^n \cap X$ for $\Sa M$-definable $X \subseteq M^n$.
\end{lem}

\section{The first inequality}
\noindent
We generalize the first inequality to arbitrary $\nip$ structures.

\begin{prop}
\label{prop:1-ineq}
Suppose that $\Sa M$ is $\nip$, $\Sa M \prec \Sa N$, and $X \subseteq N^m$ is $\Sa N$-definable.
Then $\dim_{\Sh M} M^n \cap X \leq \dim_{\Sa N} X$
\end{prop}

\noindent
Taking $X = M$ we get $\dim \Sa M = \dim \Sh M$.
So Proposition~\ref{prop:1-ineq} generalizes Fact~\ref{fact:shelah-1}.
The proof below is essentially the same as Onshuus and Usvyatsov's proof that $\Sh M$ is dp-minimal when $\Sa M$ is dp-minimal~\cite{OnUs}. 

\begin{proof}
Let $L^{\mathrm{Sh}}$ be the language of $\Sh M$.
If $\Sa N \prec \Sa O$ is highly saturated then we have $M^m \cap X = M^m \cap X(\Sa O)$, so after possibly replacing $\Sa N$ with $\Sa O$ we suppose that $\Sa N$ is highly saturated.
Let $Y := M^m \cap X$ and $\lambda$ be a cardinal.
Suppose that $\dim_{\Sh M} Y \geq \lambda$.
Let $|y| = m$ and fix a sequence $\Phi := (\varphi_\alpha(x_\alpha;y) : \alpha < \lambda )$ of $L^{\mathrm{Sh}}$-formulas such that for every finite $F \subseteq \lambda$ and $n$ there is a $(\Sh M, Y,\Phi,F,n)$-array.
By Fact~\ref{fact:shelah-cs} we have for each $\alpha \leq \lambda$ an $L$-formula $\theta_\alpha(x_\alpha;y)$ such that 
$$ \Sa M \models \varphi_\alpha(a;b) \quad \text{if and only if} \quad \Sa N \models \theta_\alpha(a;b) \quad \text{for all} \quad a \in M^{|x_\alpha|}, b \in M^{m}. $$
Fix finite $F \subseteq \lambda$ and $n$.
Let $\Theta$ be the sequence $(\theta_\alpha(x_\alpha;y) : \alpha < \lambda)$.
Observe that if $\Cal A := (a_{\alpha,i} \in M^{|x_\alpha|} : \alpha \in F, i \leq n)$ is an $(\Sh M,Y,\Phi,F,n)$-array then $\Cal A$ is also an $(\Sa N,X,\Theta,F,n)$-array.
So $\dim_{\Sa N} X \geq \lambda$.
\end{proof}

\section{The second inequality}
\label{section:results}
\noindent
We generalize the second inequality.
The results of this section are easily adapted to expansions of divisible ordered abelian groups, we leave that to the reader.
\newline

\noindent
Let $(K,\val)$ be a valued field, $\Sa K$ be an expansion of $(K,\val)$, and $\Sa K \prec \Sa L$.
Let $V$ be the set of $a \in L$ such that $\val(a) \geq \val(b)$ for some $b \in K$.
Then $\Sa L$ is a \textbf{tame extension} of $\Sa K$ if for every $a \in V$ there is $b \in K$ such that $\val(a - b) \geq \val(a - b')$ for all $b' \in K$.
It is easy to see that $b$ must be unique, so if $\Sa K \prec \Sa L$ is tame then we let $\st : L \to K$ be the map taking each $a$ to the unique $b = \st(a)$ with this property.
If $(K,\val)$ is locally compact then any elementary extension is tame.
\newline

\noindent
\textbf{In this section $\Sa K$ is $\nip$ and $\Sa K \prec \Sa L$ is tame}.
Let $\st(\Sa L)$ be the structure on $K$ with an $m$-ary relation defining $\st(V^m \cap X)$ for each $\Sa L$-definable $X \subseteq L^m$.

\begin{prop}
\label{prop:induced-0}
If $X \subseteq L^m$ is $\Sa L$-definable then $\dim_{\st(\Sa L)} \st(V^m \cap X) \leq \dim_{\Sa L} X$.
\end{prop}

\noindent
So in particular $\st(\Sa L)$ is $\nip$ and $\st(\Sa L)$ is dp-minimal when $\Sa K$ is dp-minimal.
\newline

\noindent
It is easy to see that $V$ is a subring of $L$ and if $a \in L \setminus V$ then $1/a \in V$, so $V$ is a valuation subring of $L$.
The maximal ideal $\mfrak$ of $V$ is the set of $a \in L$ such that $\val(a) \geq \val(b)$ for all $b \in K^\times$.
Observe that $\{\st(a)\} = (a + \mfrak) \cap K$ for all $a \in V$, so we may identify $K$ with $V/\mfrak$.
It is easy to see that $\st : V \to K$ is the residue map.
We describe the associated valuation.
Let $\Gamma_{K},\Gamma_{L}$ be the value group of $(K,\val), (L,\val)$, respectively.
Let $O$ be the convex hull of $\Gamma_{K}$ in $\Gamma_{L}$ and $w$ be the valuation on $L$ given by composing $\val$ with the quotient $\Gamma_{L} \to \Gamma_{K}/O$.
Then $V$ is the valuation ring of $w$.
We now prove Proposition~\ref{prop:induced-0}.

\begin{proof}
By definition $O$ is a convex subset of $\Gamma_{L}$ so $O$ is definable in $\Sh L$.
So $w$ is an $\Sh L$-definable valuation and we can regard $K$ as an imaginary sort of $\Sh L$, thus $\st : V^m \to K^m$ is $\Sh L$-definable.
The proposition now follows from Fact~\ref{fact:dp-rank-basic-1}$(3)$.
\end{proof}

\noindent
What is not clear at the moment is how $\st(\Sa L)$ relates to $\Sa K$.

\begin{prop}
\label{prop:induced-00}
$\st(\Sa L)$ is a reduct of $\Sh K$.
\end{prop}

\begin{proof}
Suppose $Y \subseteq L^m$ is $\Sa L$-definable.
Let $Z := Y + \mfrak^m$, so $Z$ is $\Sh L$-definable.
Note that $\st(V^m \cap Y) = Z \cap K^m$.
So $\st(V^m \cap Y)$ is $\Sh K$-definable by Fact~\ref{fact:external-basic}.
\end{proof}

\noindent
In general $\Sa K$ is not a reduct of $\st(\Sa L)$.
By \cite{big-nip} $\st(\Sa L)$ cannot define a subset of $\Q^m_p$ which is dense and co-dense in a nonempty open set, but there are $\nip$ expansions of $\Q_p$ which define such sets.
For example Mariaule~\cite{Ma-adic} shows that if $H$ is a dense finitely generated subgroup of $(1 + p\Z_p,\times)$ then $(\Q_p,H)$ is $\nip$.
We expect that in this case $\st(\Sa L)$ is interdefinable with $\Q_p$ but we have not carefully checked this.

\begin{prop}
\label{prop:induced-1}
Suppose that $\Sa K$ is dp-minimal.
Then $\Sa K$ is a reduct of $\st(\Sa L)$ and $\Sh K$ is interdefinable with $\st(\Sh L)$.
\end{prop}

\begin{proof}
The proof of Proposition~\ref{prop:induced-00} shows that $\st(\Sh L)$ is a reduct of $\Sh K$.
We first show that $\Sh K$ is a reduct of $\st(\Sh L)$.
Suppose that $X \subseteq K^n$ is $\Sh K$-definable.
We show that $X$ is $\st(\Sh L)$-definable.
By Facts~\ref{fact:shelah-1} and
\ref{fact:constructible} we may suppose that $X$ is closed.
Let $\Sa L \prec \Sa N$ be highly saturated, $Z \subseteq N^n$ be an honest definition of $X$, and $Y := L^n \cap Z$.
So $Y$ is $\Sh L$-definable, we show that $\st(V^n \cap Y) = X$.
As $X \subseteq Y$ and $\st$ is the identity of $K^n$ we have $X \subseteq \st(V^n \cap Y)$.
Fix $p \in \st(V^n \cap Y)$.
We show that $p \in X$.
As $X$ is closed it suffices to fix a $\val$-ball $B \subseteq K^n$ containing $p$ and show that $B \cap X \neq \emptyset$.
Fix $q \in V^n \cap Y$ such that $\st(q) = p$.
Then $q \in B(\Sa N)$ so $B(\Sa N) \cap Z\neq \emptyset$.
As $Z$ is honest $B \cap X$ is nonempty.
\newline

\noindent
It remains to show that $\Sa K$ is a reduct of $\st(\Sa K)$.
Suppose $X \subseteq K^n$ is $\Sa K$-definable.
By Fact~\ref{fact:constructible} we may suppose $X$ is closed.
Let $Y$ be the subset of $L^n$ defined by the same formula as $X$.
The proceeding paragraph shows that $\st(V^n \cap Y) = X$.
\end{proof}

\section{Delon's Theorem}
\label{section:delon}
\noindent
Fact~\ref{fact:henselian} is a well-known consequence of Pas's quantifier elimination~\cite{Pas}.

\begin{fact}
\label{fact:henselian}
Let $(M,v)$ be a Henselian valued field of equicharacteristic zero with residue field $R$.
Every $(M,v)$-definable subset every of $R^m$ is $R$-definable.
\end{fact}

\noindent
We now give a proof of Delon's theorem that $\Q_p$ is Shelah complete.

\begin{proof}
Let $\pfield \prec \Sa L$ be highly saturated.
So $\Sa L$ is a tame extension as $\pfield$ is locally compact.
Let $w$ be the valuation on $\Sa N$ with residue map $\st$.
By the observations above $w$ is a coarsening of the $p$-adic valuation on $L$, so $w$ is Henselian as a coarsening of a Henselian valuation is always Henselian.
An application of Fact~\ref{fact:henselian} shows that $\st(\Sa L)$ is interdefinable with $\pfield$.
Slight modifications to the proof of Proposition~\ref{prop:induced-1} show that $\pfield^{\mathrm{Sh}}$ and $\st(\Sa L)$ are interdefinable.
\end{proof}

\noindent
A subfield of $\Q_p$ is an elementary substructure if and only if it is algebraically closed in $\Q_p$~\cite[Lemma 6.2.1]{EP-value}.
So Corollary~\ref{cor:delon} follows from Fact~\ref{fact:delon} and Lemma~\ref{lem:completion}.

\begin{cor}
\label{cor:delon}
Suppose that $K$ is a subfield of $\Q_p$ which is algebraically closed in $\Q_p$ (e.g. the algebraic closure of $\Q$ in $\Q_p$).
Then $X \subseteq K^m$ is $K^{\mathrm{Sh}}$-definable if and only if $X = K^m \cap Y$ for some $\Q_p$-definable $Y \subseteq \Q^m_p$.
\end{cor}

\noindent
So the Shelah completion of $K$ is the structure induced on $K$ by its valuation-theoretic completion.
There are several similar results.
If $R$ is a real closed subfield of $(\R,+,\times)$ then every $R^{\mathrm{Sh}}$-definable set is of the form $R^m \cap X$ for $\rfield)$-definable $X \subseteq \R^m$.
More generally, suppose that $R$ is a divisible subgroup of $(\R,+)$ and $\Sa R$ is an o-minimal expansion of $(R,<,+)$.
Laskowski and Steinhorn~\cite{LasStein} show that there is a unique o-minimal expansion $\Sq R$ of $(\R,<,+)$ such that $\Sa R \prec \Sq R$.
By Marker-Steinhorn~\cite{Marker-Steinhorn} $\Sq R$ is Shelah complete, so every $\Sh R$-definable set is of the form $R^m \cap X$ for $\Sq R$-definable $X \subseteq \R^m$.
Finally, if $H$ is a dense subgroup of $(\R,+)$ then every $(H,+,<)^{\mathrm{Sh}}$-definable set is a boolean combination of $(H,+)$-definable sets and sets of the form $H^n \cap X$ for $(\R,<,+)$-definable $X \subseteq \R^m$, see \cite{dp-embedd}.
(If $H$ is not $n$-divisible then $nH \neq X \cap \R$ for any $(\R,<,+)$-definable $X \subseteq \R$.)


\section{$\anfield$ is Shelah complete}
\label{section:analytic}
\noindent
Let $\anfield$ be the expansion of $\Q_p$ by all analytic functions $\Z^m_p \to \Q_p$ for all $m$.
There is a well-developed theory of $\anfield$-definable sets beginning with Denef and van den Dries~\cite{Denef1988}.
It is shown in \cite{dhm-subanalytic} that every definable unary set in every elementary extension of $\anfield$ is definable in the underlying field.
Fact~\ref{fact:an-dp} easily follows.

\begin{fact}
\label{fact:an-dp}
$\anfield$ is dp-minimal.
\end{fact}

\noindent
So in particular dp-rank agrees with the canonical dimension on $\anfield$-definable sets.
\newline

\noindent
Suppose that $\anfield \prec \Sa L$ is highly saturated and let $L$ be the underlying field of $\Sa L$.
Note that $\anfield \prec \Sa L$ is tame.
Let $\valp$ be the $p$-adic valuation, $V$ be the set of $a \in L$ such that $\valp(a) \geq k$ for some $k$, and $\st : V^m \to \Q^m_p$ be the standard part map.
As above $V$ is a valuation subring of $L$, the associated valuation is a coarsening of $\valp$, and $\st : V \to \Q_p$ is the residue map.
Fact~\ref{fact:sr} is the analytic analogue of Fact~\ref{fact:henselian}.
Fact~\ref{fact:sr} follows easily from a theorem of Rideau~\cite[Theorem 3.10]{rideau-analytic-difference}.
(This is closely related to the work of Cluckers, Lipshitz, and Robinson on the model theory of valued fields with analytic structure \cite{cluckers-lipshitz-analytic,MR3584560,MR2290137}.)

\begin{fact}
\label{fact:sr}
A subset $X$ of $\Q^m_p$ is $(\Sa L,V)$-definable if and only if it is $\anfield$-definable.
\end{fact}

\noindent
Following the argument of Section~\ref{section:delon}, applying Fact~\ref{fact:an-dp} when necessary, and applying Fact~\ref{fact:sr} in place of Fact~\ref{fact:henselian} we obtain Theorem~\ref{thm:delon-an}.

\begin{thm}
\label{thm:delon-an}
If $X \subseteq L^m$ is $\Sa L$-definable then $\Q^m_p \cap X$ is $\anfield$-definable.
\end{thm}

\noindent
So $\anfield$ is Shelah complete.
(This was proven in unpublished work of Hrushovski, see \cite[Fact 2.6]{Onshuus2008}.)
Proposition~\ref{prop:an-dim-eq} follows by Proposition~\ref{prop:induced-0} and Fact~\ref{fact:sr}.

\begin{prop}
\label{prop:an-dim-eq}
If $X$ is an $\Sa L$-definable subset of $L^m$ then $\dim_{\anfield} \st(X) \leq \dim_{\Sa L} X$.
\end{prop}

\noindent
Proposition~\ref{prop:induced-1} and Theorem~\ref{thm:delon-an} together yield a strengthening of Fact~\ref{fact:sr}.

\begin{cor}
\label{cor:sr}
The structure induced on $\Q_p$ by $\Sh L$ is interdefinable with $\anfield$.
\end{cor}

\section{A geometric application}
\label{section:geometric}
\noindent
Following work of Br\"{o}cker~\cite{MR1226248,MR1260702} and van den Dries~\cite{vdd-limit} in the semialgebraic and o-minimal settings, respectively, we give a geometric application of the results above.
We let $|,|$ be the usual absolute value on $\Q_p$ and declare
$$ \|a\| := \max \{|a_1|,\ldots,|a_m| \}  \quad \text{for all}  \quad a = (a_1,\ldots,a_m) \in \Q^m_p. $$
The Hausdorff distance $d_H(X,X')$ between bounded subsets $X,X'$ of $\Q^m_p$ is the infimum of $t \in \R_{>0}$ such that for every $a \in X$ there is $a' \in X'$ such that $\| a - a' \| < t$ and for every $a' \in X'$ there is $a \in X$ such that $\| a - a' \| < t$.
The Hausdorff distance between a bounded set and its closure is always zero.
If $\Cal X$ is a family of bounded subsets of $\Q^m_p$ then $X \subseteq \Q^m_p$ is a \textbf{Hausdorff limit} of $\Cal X$ if $X$ is compact and there is a sequence $(X_k)_{k \in \N}$ of elements of $\Cal X$ such that $d_H(X_k,X) \to 0$ as $k \to \infty$.

\begin{thm}
\label{thm:hausdorff-limit}
Suppose that $\Cal X$ is an $\anfield$-definable family of bounded subsets of $\Q^m_p$.
Any Hausdorff limit of $\Cal X$ is $\anfield$-definable.
If $(X_k)_{k \in \N}$ is a sequence of elements of $\Cal X$ which Hausdorff converges to $X \subseteq \Q^m_p$ then $\dim X \leq \limsup_{k \to \infty} \dim X_k$.
\end{thm}

\noindent
Note that \textit{any} compact subset of $\Z^m_p$ is a Hausdorff limit of a sequence of finite sets so the restriction to definable families of sets is necessary.
We will need to use Fact~\ref{fact:def-dim}, an immediate consequence of the equality of naive dimension and the canonical dimension on $\anfield$-definable sets.

\begin{fact}
\label{fact:def-dim}
Suppose that $(X_a :a \in \Q^n_p )$ is an $\anfield$-definable family of subsets of $\Q^m_p$.
Then $\{ a \in \Q^n_p : \dim X_a = l \}$ is $\anfield$-definable for any $0 \leq l \leq m$.
\end{fact}

\noindent
We now proceed to prove Theorem~\ref{thm:hausdorff-limit}.
Our proof is very similar to that in \cite{vdd-limit} so we omit some details.
We also make a nonessential use of ultrafilter convergence.

\begin{proof}
Let $\anfield \prec \Sa L$ be highly saturated.
Let $|x| = m$ and $\phi(x;y)$ be a formula such that $\Cal X$ is  $(\phi(\Q^m_p;b) : b \in \Q^{|y|}_p)$.
For each $k$ fix $b_k \in \Q^{|y|}_p$ such that $X_k = \phi(\Q^m_p;b_k)$.
Let $\mathfrak{u}$ be a nonprinciple ultrafilter on $\N$.
Applying saturation fix $b \in K^{|y|}$ such that $\tp(b_k|\Q_p) \to \tp(b|\Q_p)$ as $k \to \mathfrak{u}$.
Let $Y := \phi(L^m;b)$.
It is easy to see that $Y \subseteq V^m$ and $X = \st(Y)$.
So $X$ is $\anfield$-definable by Fact~\ref{fact:sr}.
By Fact~\ref{fact:def-dim} we have
$$\dim_{\Sa L} Y = \lim_{k \to \mathfrak{u}} \dim X_k \leq \limsup_{k \to \infty} \dim X_k.$$
So by Proposition~\ref{prop:an-dim-eq} we have $\dim X \leq \limsup_{k \to \infty} \dim X_k$.
\end{proof}

\noindent
Our proof of Theorem~\ref{thm:hausdorff-limit} goes through over $\Q_p$.
We give an attractive formulation of the first claim in this setting.
For each $k \geq 0$ we let $P_k$ be a unary relation defining the set of $k$th powers in $\Q_p$.
It is a famous theorem of Macintyre~\cite{macintyre-p-adic} that every paremeter free formula in the language of rings is equivalent over $\Q_p$ to a boolean combination of formulas of the form $f = g$, $\valp(f) \leq \valp(g)$, or $P_k(f)$ for $f,g \in \Z[x_1,\ldots,x_m]$.
Suppose $X \subseteq \Q^m_p$ is $\Q_p$-definable.
The \textbf{complexity} of $X$ is $\leq n$ if $X$ may be defined using $\leq n$ formulas of the form $f = g, \valp(f) \leq \valp(g)$, or $P_k(f)$ where each $k \leq n$ and each $f,g \in \Q_p[x_1,\ldots,x_m]$ has degree $\leq n$.
It is easy to see that Theorem~\ref{thm:qp-complexity}  follows from saturation and the fact that a Hausdorff limit of a sequence of elements of a $\Q_p$-definable family of sets is $\Q_p$-definable.

\begin{thm}
\label{thm:qp-complexity}
For every $n,m$ there is an $l$ such that if $(X_k)_{k \in \N}$ is a Hausdorff converging sequence of bounded $\Q_p$-definable subsets of $\Q^m_p$, each of complexity $\leq n$, then the Hausdorff limit $X$ of $(X_k)_{k \in \N}$ is $\Q_p$-definable of complexity $\leq l$.
\end{thm}


\noindent
Proposition~\ref{prop:other-direction} shows that Theorem~\ref{thm:hausdorff-limit} is equivalent to the fact that $\st(X)$ is $\anfield$-definable when $X \subseteq V^m$ and Proposition~\ref{prop:an-dim-eq}.
 
\begin{prop}
\label{prop:other-direction}
Let $\anfield \prec \Sa L$ be highly saturated.
Suppose that $\phi(x;y)$ is a formula in the language of $\anfield$.
Fix $b \in L^{|y|}$ and suppose  $X := \phi(L^{|x|};b) \subseteq V^{|x|}$.
Then $\st(X)$ is a Hausdorff limit of $\Cal X := (\phi(\Q^{|x|}_p;a) : a \in \Q^{|y|}_p)$.
\end{prop}

\noindent
Let $|x| = m$ and $|y| = n $.
Given subsets $X,X'$ of $V^m$ and $t \in \R_{>0}$ we say that $d_H(X,X') < t$ if for every $a \in X$ there is $a \in X'$ such that $\| a - a' \| < t$ and vice versa (we do not define $d_H(X,X')$ in this case).

\begin{proof}
By saturation $\st(X)$ is compact so it suffices to show that for every $t \in \R_{>0}$ there is a bounded $Y \in \Cal X$ such that $d_H(\st(X),Y) \leq t$.
Fix $t \in \R_{>0}$.
As $X \subseteq V^m$ it is easy to see there is a finite $A \subseteq \Q^m_p$ such that $d_H(A,X) < t/2$, observe that $d_H(A,\st(X)) \leq t/2$.
As $\anfield$ is an elementary submodel of $\Sa K$ we obtain $a \in \Q^n_p$ such that $d_H(\phi(\Q^m_p;a),A) < t/2$.
Let $Y := \phi(\Q^m_p;a)$.
Note that $Y$ is bounded.
The triangle inequality for $d_H$ yields $d_H(\st(X),Y) \leq t$.
\end{proof}

\noindent
It should be possible to give a geometric proof of  Theorem~\ref{thm:hausdorff-limit} and thereby obtain a geometric proof of Fact~\ref{fact:sr}.
We are aware of two geometric proofs of the o-minimal analogue of Theorem~\ref{thm:hausdorff-limit}, Lion and Speissegger~\cite{MR2099073} and Kocel-Cynk, Pawlucki,  and Valette~\cite{MR3159091}.
The main tool of \cite{MR3159091} is the o-minimal Lipschitz cell decomposition, see~\cite{MR2546024}, and there is now a Lipschitz cell decomposition for $\anfield$-definable sets~\cite{MR2647135}.
We believe that there is a purely geometric proof of Shelah completeness for $\Q_p$ and $\anfield$ along these lines, but we have not seriously pursued this.
   
\section{A question}
\noindent
Fix an o-minimal expansion $\Sa R$ of $\rfield)$ such that the function $\R_{>0} \to \R_{>0}$ given by $t \mapsto t^r$ is only definable when $r \in \Q$, e.g.  $\R_{\mathrm{an}}$.
Fix $\lambda \in \R_{>0}$ and let $\lambda^\Z := \{ \lambda^m : m \in \Z\}$.
Following \cite{vdd-Powers2} Miller and Speissegger show that $\slam$ is tame~\cite[Section 8.6]{Miller-tame}.
It follows by \cite[Theorem 4.1.2, Corollary 4.1.7]{Tychon-thesis} and \cite[Corollary 2.6]{CS-I} that $\slam$ is $\nip$.
There is a canonical notion of dimension $d$ for $\slam$-definable sets which agrees with naive, topological, and Assouad dimension.
\newline

\noindent
Let $\slam \prec \Sa N$, $V$ be the convex hull of $\R$ in $H$, $\st$ be the standard part map $V^m \to \R^m$, and $X \subseteq N^m$ be $\Sa N$-definable.
We believe $\slam$ is Shelah complete but we have not carefully checked this.
Assuming that this is true, it is possible to give a geometric proof that $d(\st(V^m \cap X)) \leq d(X)$.
There should be a $\nip$-theoretic proof.
More specifically there should be a combinatorical invariant $\mathbf{I}_{\Sa M} Y$ of a definable set $Y$ in an $\nip$ structure $\Sa M$ which satisfies at least the following:
\begin{enumerate}
\item $\mathbf{I}_{\slam}$ agrees with $d$,
\item $\mathbf{I}_{\Sh M} Y = \mathbf{I}_{\Sa M} Y$ (this should be immediate from the definition and Fact~\ref{fact:shelah-cs}),
\item $\mathbf{I}_{\Sa M} Y \cup Y' = \max\{ \mathbf{I}_{\Sa M} Y, \mathbf{I}_{\Sa M} Y'\}$
\item $\mathbf{I}_{\Sa M} Y \times Y' = \mathbf{I}_{\Sa M} Y + \mathbf{I}_{\Sa M} Y'$,
\item $\mathbf{I}_{\Sa M} f(Y) \leq \mathbf{I}_{\Sa M} Y$ for any $\Sa M$-definable function $f : Y \to M^n$.
\end{enumerate}
If $\mathbf{I}$ satisfies $(1) - (5)$ and $\slam$ is Shelah complete then we have
$$d(\st(V^m \cap X)) = \mathbf{I}_{\slam} \st(V^m \cap X) \leq \mathbf{I}_{\Sh N} X = \mathbf{I}_{\Sa N} X = d(X).$$
Dp-rank does not satisfy $(1)$.
Tychonievich~\cite{Tychon-thesis} has shown that every countable $\slam$-definable set is internal to $\lambda^\Z$ and the induced structure on $\lambda^\Z$ is interdefinable with $(\lambda^\Z,\times,<)$.
So the dp-rank agrees with the canonical Presburger dimension on countable $\slam$-definable sets.
Furthermore the dp-rank of any uncountable $\slam$-definable set is $\aleph_0 - 1$.
So in this setting dp-rank is not very useful.
\newline

\noindent
Suppose $Z \subseteq \R^m$ is $\slam$-definable.
If $d(Z) = 0$ then $Z$ is internal to $\lambda^\Z$ and if $d(Z) > 0$ then there is a definable surjection $Z \to \R$.
As $(\Z,+,<)$ does not interpret an infinite field (see \cite{big-nip}) we have $d(Z) > 0$ if and only if the induced structure on $Z$ does not interpret an infinite field.
We should have $d(Z) = 0$ if and only if $Z$ is ``modular".
So $\mathbf{I}_{\Sa M} X$ should be the ``non-modular dimension" of $X$.
\newline

\noindent
We give two other examples of $\nip$ structures to which this should apply.
The first example is $(\Sa S, \{ \lambda, \lambda^\lambda,\lambda^{\lambda^\lambda},\ldots \})$ where $\lambda > 1$ and $\Sa S$ is an o-minimal expansion of $\rfield)$ such that every $\Sa S$-definable function $\R \to \R$ is eventually bounded above by some compositional iterate of the exponential (all known o-minimal expansions of $\rfield)$ satisfy this condition).
Miller and Tyne~\cite{Miller-iteration} show that this structure is tame, in particular naive dimension is well behaved.
The induced structure on $D := \{\lambda, \lambda^\lambda,\lambda^{\lambda^\lambda},\ldots \}$ should be interdefinable with $(D,<)$ and any zero-dimensional definable set should be internal to $D$.
So $\mathbf{I}_{(\Sa S,D)}$ should agree with naive dimension.
Second, let $\Log$ be the Iwasawa logarithm $\Q^\times_p \to \Q_p$.
Mariaule~\cite{Ma-adic} shows that $(\Q_p,\Log)$ is $\nip$.
We have $\Log(a) = 0$ if and only if $a = b p^k$ for some $k$ and root of unity $b \in \Q_p$.
It follows that $p^\Z$ is $(\Q_p,\Log)$-definable.
It is also shown in \cite{Ma-adic} that the induced structure on $p^\Z$ is interdefinable with $(p^\Z,\times,\vartriangleleft)$ where $p^k \vartriangleleft p^l$ if and only if $k < l$.
It should follow from \cite{Ma-adic} that naive dimension is well behaved in $(\Q_p,\Log)$ and a zero-dimensional definable set should be internal to $p^\Z$.
So we expect $\mathbf{I}_{(\Q_p,\Log)}$ to coincide with naive dimension.
\newline

\noindent
The vague question here is: What is the right combinatorial definition of the canonical dimension in structures such as $(\Sa R,\lambda^\Z)$, $(\Sa S, \{\lambda, \lambda^\lambda,\lambda^{\lambda^\lambda},\ldots \})$, or $(\Q_p,\Log)$?

\bibliographystyle{abbrv}
\bibliography{NIP}
\end{document}